\documentclass{amsart}%
\usepackage{amsfonts}
\usepackage{amsmath}
\usepackage{amssymb}
\usepackage{graphicx}%
\providecommand{\U}[1]{\protect\rule{.1in}{.1in}}
\newtheorem{theorem}{Theorem}

\newtheorem{corollary}[theorem]{Corollary}

\newtheorem{definition}[theorem]{Definition}

\newtheorem{lemma}[theorem]{Lemma}

\newtheorem{proposition}[theorem]{Proposition}
\newtheorem{remark}[theorem]{Remark}

\begin{document}
\title{An $A_{\infty}$-coalgebra Structure on a Closed Compact Surface}
\author{Quinn Minnich$^{1}$}
\address{Millersville University of Pennsylvania\\
Department of Mathematics \\
Millersville, PA 17551}
\author{Ronald Umble}
\email{ron.umble@millersville.edu}
\thanks{$^{1}$ The main result (Theorem \ref{Main Result}) was proved by the first
author in his senior honors thesis \cite{Minnich}.}
\date{May 2, 2018}

\begin{abstract}
Let $P$ be an $n$-gon with $n\geq3.$ There is a formal combinatorial $A_\infty$-coalgebra structure on cellular chains $C_*(P)$ with non-vanishing higher order structure when $n\geq5$. If $X_g$ is a closed compact surface of genus $g\geq2$ and $P_g$ is a polygonal decomposition, the quotient map $q:P_g\to X_g$ projects the formal $A_\infty$-coalgebra structure on $C_*(P_g)$ to a quotient structure on $C_*(X_g)$, which persists to homology $H_{\ast}\left(  X_g;\mathbb{Z}_{2}\right) $, whose operations are determined by the quotient map $q$, and whose higher order structure is non-trivial if and only if $X_g$ is orientable or unorientable with $g\geq3$. But whether or not the $A_{\infty}$-coalgebra structure on homology observed here is topologically invariant is an open question.
\vspace{.2in}
\begin{center}
\textit{To Tornike Kadeishvili on the occasion of his 70th birthday} 
\end{center}
\end{abstract}
\maketitle

\section{Introduction}

Let $R$ be a commutative ring with unity and let $P$ be an $n$-gon with $n\geq3$. In this paper we construct a formal combinatorial $A_{\infty}$-coalgebra structure on the cellular chains of $P$, denoted by $C_{\ast}(P)$, which is the graded $R$-module generated by the vertices, edges, and region of $P$. For an application,
let $X_g$ be a closed compact surface of genus $g\geq2$ and let $P_g$ be a polygonal decomposition. The quotient map $q:P_g\to X_g$ sends the formal $A_\infty$-coalgebra structure on $C_*(P_g)$ to a quotient structure on $C_*(X_g)$, which persists to homology $H_{\ast}\left(  X_g;\mathbb{Z}_{2}\right) $, whose operations are determined by the quotient map $q$, and whose higher order structure is non-trivial if and only if $X_g$ is orientable or unorientable with $g\geq3$.

An $A_{\infty}$-coalgebra is the linear dual of an $A_{\infty}$-algebra
defined by J. Stasheff \cite{Stasheff} in the setting of base pointed loop
spaces. As motivation, we begin with a brief description of $A_{\infty}$-algebras.

Let $S$ be a surface embedded in $\mathbb{R}^{3}$ and let $\ast$ be some
specified base point on $S.$ A \emph{base pointed loop on} $S$ is a continuous
map $\alpha:I\rightarrow S$ such that $\alpha\left(  0\right)  =\alpha\left(
1\right)  =\ast.$ Let $\Omega S$ denote the space of all base pointed loops on
$S.$ Given $\alpha,\beta\in\Omega S,$ define their \emph{product} $\alpha
\cdot\beta\in\Omega S$ to be
\[
\left(  \alpha\cdot\beta\right)  \left(  t\right)  =\left\{
\begin{array}
[c]{ll}%
\alpha\left(  2t\right)  , & t\in[0,\frac{1}{2}]\\
\beta\left(  2t-1\right)  , & t\in[\frac{1}{2},1].
\end{array}
\right.
\]
A \emph{homotopy from }$\alpha$ \emph{to} $\beta$ is a continuous map
$H:I\rightarrow\Omega S$ such that $H\left(  0\right)  =\alpha$ and $H\left(
1\right)  =\beta.$ Thus $\left\{  H\left(  s\right)  :s\in I\right\}  $ is a
$1$-parameter family of loops that continuously deforms $\alpha$ to $\beta$.

Let $\alpha,\beta,\gamma\in\Omega S$. Although $\left(  \alpha\cdot
\beta\right)  \cdot\gamma\neq\alpha\cdot\left(  \beta\cdot\gamma\right)  ,$
the loops $\left(  \alpha\cdot\beta\right)  \cdot\gamma$ and $\alpha
\cdot\left(  \beta\cdot\gamma\right)  $ are homotopic via a linear change of
parameter homotopy $H$. Let $\mathbf{1}:\Omega S\rightarrow\Omega S$ be the
identity map and define $m_{2}:\Omega S\otimes\Omega S\rightarrow\Omega S$ by
$m_{2}\left(  \alpha\otimes\beta\right)  =\alpha\cdot\beta.$ Then
$m_{2}\left(  m_{2}\otimes\mathbf{1}\right)  \left(  \alpha\otimes\beta
\otimes\gamma\right)  =\left(  \alpha\cdot\beta\right)  \cdot\gamma$ and
$m_{2}\left(  \mathbf{1}\otimes m_{2}\right)  \left(  \alpha\otimes
\beta\otimes\gamma\right)  =\alpha\cdot\left(  \beta\cdot\gamma\right)  .$
Consider $m_{2}\left(  m_{2}\otimes\mathbf{1}\right)  ,m_{2}\left(
\mathbf{1}\otimes m_{2}\right)  :\Omega S^{\otimes3}\rightarrow\Omega S$ and
think of the homotopy $H$ from $\left(  \alpha\cdot\beta\right)  \cdot\gamma$
to $\alpha\cdot\left(  \beta\cdot\gamma\right)  $ as a $3$-ary operation
$m_{3}:\Omega S^{\otimes3}\rightarrow\Omega S.$ Identify $m_{3}$ with the
interval $I,$ its endpoint $0$ with $m_{2}\left(
m_{2}\otimes\mathbf{1}\right)  ,$ and its endpoint $1$ with $m_{2}\left(
\mathbf{1}\otimes m_{2}\right)  .$ Then the boundary $\partial m_{3}%
=m_{2}\left(  \mathbf{1}\otimes m_{2}\right)  -m_{2}\left(  m_{2}%
\otimes\mathbf{1}\right)  $ and the parameter space $I  $
identified with $m_{3}$ is called the associahedron $K_{3}.$ Thus $K_{3}$
controls homotopy associativity in three variables.

In a similar way, homotopy associativity in four variables is controlled by
the associahedron $K_{4},$ which is a pentagon. The vertices of $K_{4}$ are
identified with the five ways one can parenthesize four variables, its edges
are identified with the homotopies that preform a single shift of parentheses,
and its $2$-dimensional region is identified with a $4$-ary operation $m_{4}:$
$\Omega S^{\otimes4}\rightarrow\Omega S.$ Thus $\partial m_{4}=m_{2}\left(
m_{3}\otimes\mathbf{1}\right)  -m_{3}\left(  m_{2}\otimes\mathbf{1}%
\otimes\mathbf{1}\right)  +m_{3}\left(  \mathbf{1}\otimes m_{2}\otimes
\mathbf{1}\right)  -m_{3}\left(  \mathbf{1}\otimes\mathbf{1}\otimes
m_{2}\right)  +m_{2}\left(  \mathbf{1}\otimes m_{3}\right)  .$ In general,
homotopy associativity in $n$ variables is controlled by the associahedron
$K_{n},$ which is an $n-2$ dimensional polytope whose vertices are identified
with the various ways to parenthesize $n$ variables. While associahedra are
independently interesting geometric objects, they also organize the data in
the definition of an $A_{\infty}$-(co)algebra.

\begin{center}
\includegraphics[height = 5cm]{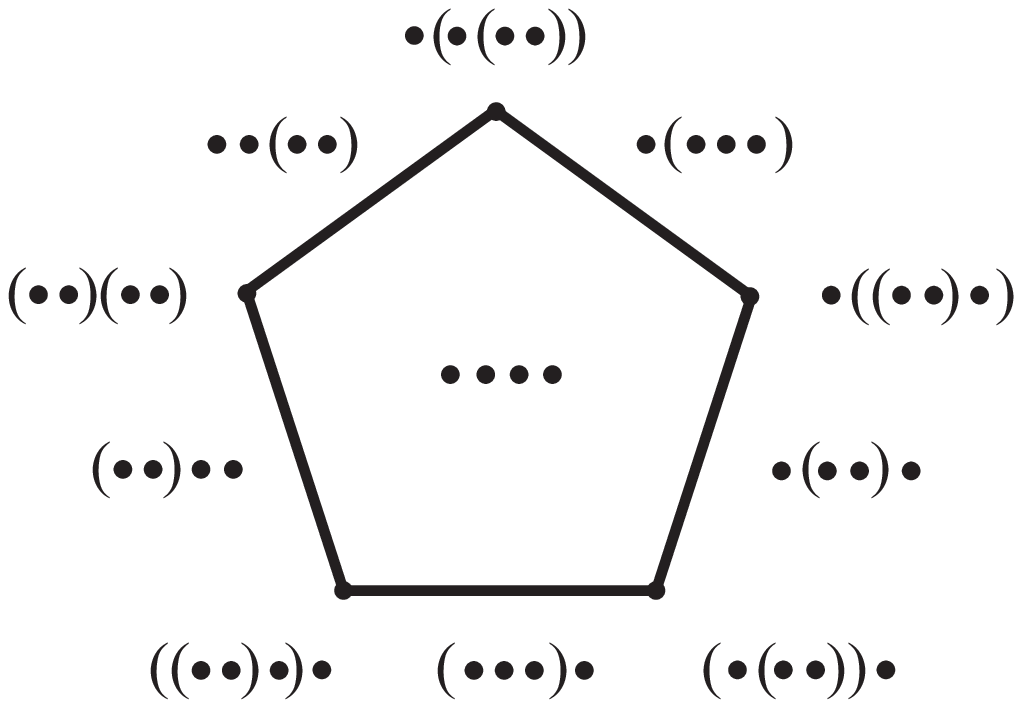} \\[0pt]Figure 1. The associahedron
$K_{4}.$
\end{center}

But before we can define an $A_{\infty}$-(co)algebra, we need some
preliminaries. A \emph{differential graded (d.g.) $R$-module} is a graded $R$-module 
$V=\oplus_{i\geq0}V_{i}$ equipped with a differential operator
$\partial:V_{\ast}\rightarrow V_{\ast-1}$ such that $\partial\circ\partial=0.$
Let $\left(  V,\partial_{V}\right)  $ and $\left(  W,\partial_{W}\right)  $ be
d.g. $R$-modules. A linear map $f:V\rightarrow W$ has \emph{degree} $p$ if
$f:V_{i}\rightarrow W_{i+p};$ the map $f$ is a \emph{chain map} if
$f\circ\partial_{V}=\left(  -1\right)  ^{p}\partial_{W}\circ f.$ Denote the
degree of $f$ by $\left\vert f\right\vert $ and the $R$-module of all linear
maps of degree $p$ by $Hom_{p}\left(  V,W\right)  $.

\begin{proposition}
$Hom_{\ast}\left(  V,W\right)  $ is a d.g. $R$-module with differential
$\delta$ given by $\delta\left(  f\right)  =f\circ\partial_{V}-\left(
-1\right)  ^{|f|}\partial_{W}\circ f.$
\end{proposition}

\begin{proof}
The proof is straight-forward and omitted.
\end{proof}

\noindent Note that $f$ is a chain map if and only if $\delta(f)=0$. Now
$m_{2}\in Hom_{\ast}( C_{\ast}\left(  P\right)  ^{\otimes2},C_{\ast}\left(
P\right)  ) $ and $m_{3}\in Hom_{\ast}( C_{\ast}\left(  P\right)  ^{\otimes
3},C_{\ast}\left(  P\right)  ) $. Since $|m_{3}|=1$ we have
\[
\delta(m_{3})=m_{3}\circ\partial^{\otimes3}-\left(  -1\right)  ^{1}%
\partial\circ m_{3}=\partial\circ m_{3}=m_{2}(\mathbf{1}\otimes m_{2}%
)-m_{2}(m_{2}\otimes\mathbf{1}),
\]
where $m_{3}\circ\partial^{\otimes3}=0$ because loops have empty boundary.
Then $\delta(m_{3})$ measures the \emph{deviation of} $m_{2}$ \emph{from
associativity}, and in certain situations one can express this deviation in
terms of a degree $0$ chain map $\alpha_{3}:C_{\ast}\left(  K_{3}\right)
\rightarrow Hom_{\ast}( C_{\ast}\left(  P\right)  ^{\otimes3},C_{\ast}\left(
P\right)  ) $. Let $\theta_n$ denote the top dimensional cell of $K_n$.

\begin{definition}
Let $(V,\partial)$ be a d.g. $R$-module. For each $n\geq2$, choose a map
$\alpha_{n}:C_{\ast}(K_{n})\rightarrow Hom(V^{\otimes n},V)$ of degree $0,$
and let $m_{n}=\alpha_{n}(\theta_{n}).$ Then $(V,\partial,m_{2}%
,m_{3},\dots)$ is an $A_{\infty}$-\textbf{algebra} if 
$\delta\alpha_n=\alpha_n\partial$ for each $n\geq2$.
\end{definition}

The definition of an $A_{\infty}$-coalgebra mirrors the definition of an
$A_{\infty}$-algebra.

\begin{definition}\label{defn3}
Let $\left(  V,\partial\right)  $ be a d.g $R$-module. For each $n\geq2$,
choose a map $\alpha_{n}:C_{\ast}\left(  K_{n}\right)  \rightarrow Hom\left(
V,V^{\otimes n}\right)  $ of degree $0,$ and let $\Delta_{n}=\alpha
_{n}\left(  \theta_{n}\right)  .$ Then $\left(  V,\partial,\Delta_{2}%
,\Delta_{3},\ldots\right)  $ is an $A_{\infty}$-\textbf{coalgebra} if 
$\delta\alpha_n  =\alpha_n\partial$ for each
$n\geq2.$
\end{definition}

\noindent Evaluating both sides of the equation in Definition \ref{defn3}
at $\theta_n$ produces the classical structure relation\bigskip

$\Delta_{n}\partial-\left(  -1\right)  ^{n-2}%
%TCIMACRO{\dsum \limits_{i=0}^{n-1}}%
%BeginExpansion
{\displaystyle\sum\limits_{i=0}^{n-1}}
%EndExpansion
\left(  \mathbf{1}^{\otimes i}\otimes\partial\otimes\mathbf{1}^{\otimes n-i-1}\right)
\Delta_{n}$
\begin{equation}\label{relation}
=\sum_{i=1}^{n-2}\sum_{j=0}^{n-i-1}\left(  -1\right)  ^{i\left(  j+n+1\right)
}\left(  \mathbf{1}^{\otimes j}\otimes\Delta_{i+1}\otimes\mathbf{1}^{\otimes n-i-j-1}\right)
\Delta_{n-i},
\end{equation}
which expresses $\Delta_{n}$ as a chain homotopy among the quadratic
compositions encoded by the codimension $1$ cells of $K_{n}$ (see Figure 2).
The sign $\left(  -1\right)  ^{i\left(  j+n+1\right)  }$ is the combinatorial
sign derived by Saneblidze and Umble in \cite{Saneblidze}. Note that when
$n=2$, Relation \ref{relation} has the form $\Delta_{2}\partial=\left(
\partial\otimes\mathbf{1}+\mathbf{1}\otimes\partial\right)  \Delta_{2}$, which
is dual to the Leibniz Rule $dm=m\left(  d\otimes\mathbf{1}+\mathbf{1}\otimes d\right)  $ in
calculus and says that $\partial$ is a coderivation of $\Delta_{2}.$

\bigskip
%TCIMACRO{\FRAME{ftbpF}{2.9888in}{1.7521in}{0pt}{}{}{k4-compositions.eps}%
%{\special{ language "Scientific Word";  type "GRAPHIC";
%maintain-aspect-ratio TRUE;  display "USEDEF";  valid_file "F";
%width 2.9888in;  height 1.7521in;  depth 0pt;  original-width 4.5567in;
%original-height 2.6558in;  cropleft "0";  croptop "1";  cropright "1";
%cropbottom "0";  filename 'K4-compositions.eps';file-properties "XNPEU";}}}%
%BeginExpansion
\begin{center}
\includegraphics[
height=1.7521in,
width=2.9888in
]%
{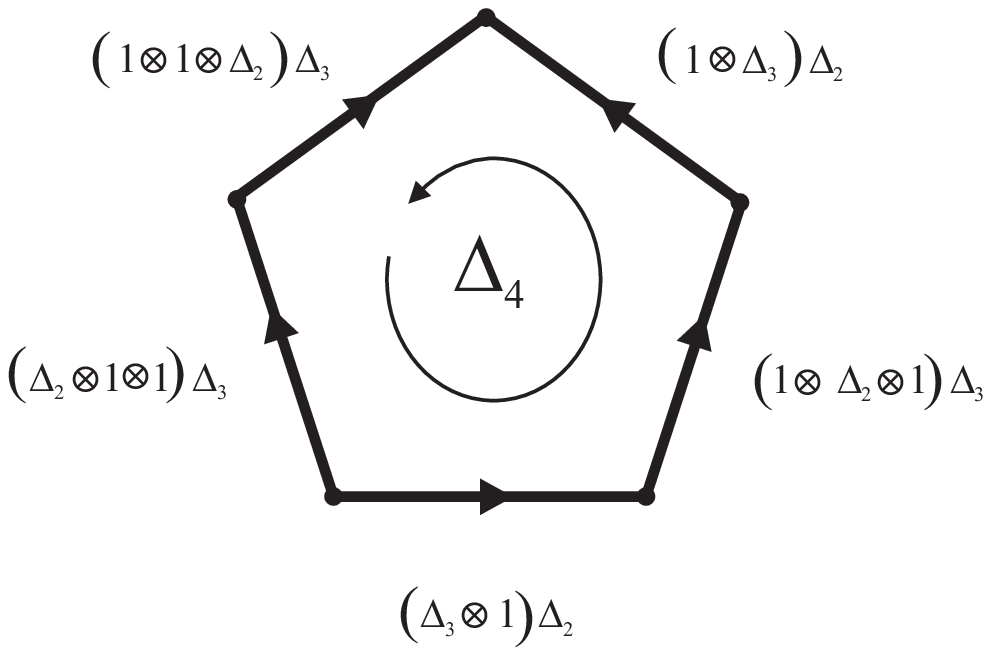}%
%EndExpansion

Figure 2. The quadratic compositions encoded by the codim $1$ cells of $K_{4}$.
\end{center}

\section{Statement of the Main Result}

Let $P$ be a counterclockwise oriented $n$-gon, $n\geq3$. Label the vertices
$v_{1},v_{2},\ldots, v_{n}$ and the edges $e_{1},e_{2},\ldots, e_{n}$. Define
$v_{1}$ to be the \emph{initial vertex} and $v_{n}$ to be the \emph{terminal
vertex}, and direct the edges from $v_{1}$ to $v_{n}$. This assignment
partially orders the vertices as indicated in Figure 3. Edges whose direction
is consistent with orientation are positive. Let $\partial:C_{\ast}\left(
P\right)  \rightarrow C_{\ast-1}\left(  P\right)  $ be the boundary operator
induced by geometric boundary. Then $\partial(v_{i})=0$, $\partial
(e_{i})=v_{i+1}-v_{i}$ if $i<n$, $\partial(e_{n})=v_{n}-v_{1}$, and
$\partial(p)=e_{1}+e_{2}+\cdots+e_{n-1}-e_{n}$.

\begin{center}
\includegraphics[height = 5cm, width = 5.5cm]{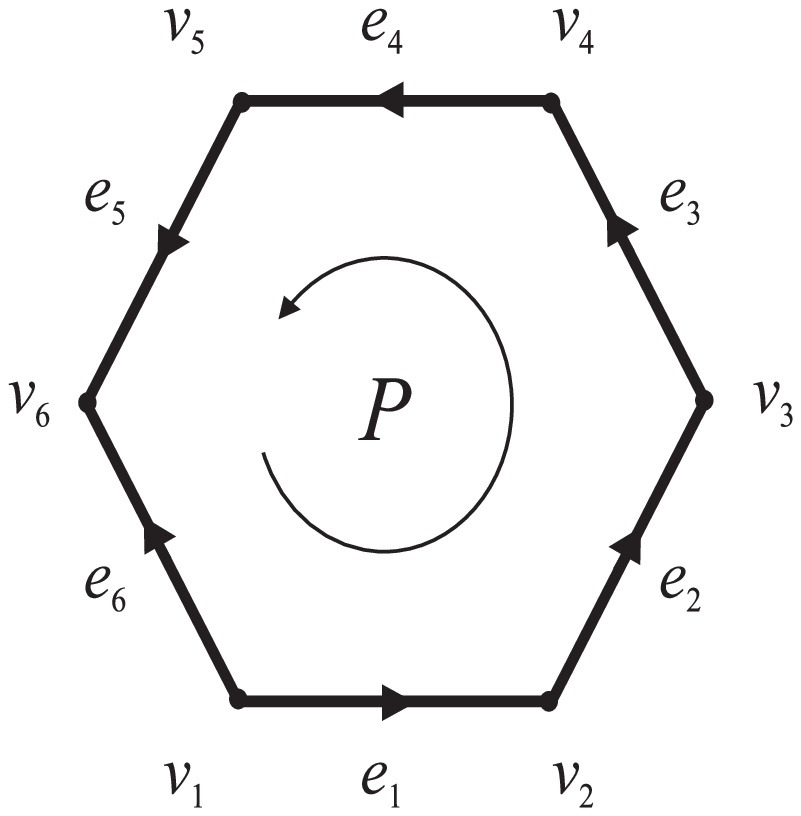} \\[0pt]Figure 3.
An $n$-gon for $n=6$
\end{center}

We will use the diagonal approximation $\Delta_{2}:C_{\ast}\left(  P\right)
\rightarrow C_{\ast}\left(  P\right)  \otimes C_{\ast}\left(  P\right)  $
defined by D. Kravatz in \cite{Kravatz} and given by%
\begin{align*}
&  \Delta_{2}(P)=v_{1}\otimes P+P\otimes v_{n}+\sum_{0<i_{1}<i_{2}<n}e_{i_{1}%
}\otimes e_{i_{2}},\\
&  \Delta_{2}(e_{i})=v_{i}\otimes e_{i}+e_{i}\otimes v_{i+1}\text{ if }i<n,\\
&  \Delta_{2}(e_{n})=v_{1}\otimes e_{n}+e_{n}\otimes v_{n},\\
&  \Delta_{2}(v_{i})=v_{i}\otimes v_{i}.
\end{align*}
A more general exposition of Kravatz's diagonal appears in \cite{Gonzalay}.
For $k>2,$ define the $k$-ary $A_{\infty}$-coalgebra operation $\Delta
_{k}:C_{\ast}\left(  P\right)  \rightarrow C_{\ast}\left(  P\right)  ^{\otimes
k}$ by
\[
\Delta_{k}(\sigma)=\left\{
\begin{array}
[c]{cl}%
\sum\limits_{0<i_{1}<i_{2}<\cdots<i_{k}<n}e_{i_{1}}\otimes e_{i_{2}}%
\otimes\cdots\otimes e_{i_{k}}, & \text{if }\sigma=P,\\
0, & \text{otherwise.}%
\end{array}
\right.
\]
Then by definition, $\Delta_{k}=0$ for all $k\geq n.$ We can now state our
main result. 

\begin{theorem}
\label{Main Result}Let $P$ be an $n$-gon. The structure operations $\left\{
\Delta_{k}\right\}  _{2\leq k<n}$ defined above impose an $A_{\infty}%
$-coalgebra structure on $\left(  C_{\ast}\left(  P\right)  ,\partial\right)
.$ Furthermore, $\Delta_{k}$ vanishes for all $k\geq n$.
\end{theorem}

\section{Proof of the Main Result}

We must verify Relation \ref{relation} for all $k\geq2.$ When $k=2$,
verification is easy and left to the reader. We first verify Relation
\ref{relation} for all $k>3$, then consider the special case $k=3.$ To
simplify notation we establish following notational devices:

\begin{itemize}
\item The symbol $\partial\Delta_{k}(P)$ denotes $\sum_{i=0}^{k-1}%
(\mathbf{1}^{\otimes i}\mathbf{\otimes}\partial\otimes\mathbf{1}^{\otimes
k-i-1})\Delta_{k}(P).$

\item The symbol $\Delta_{j}\Delta_{k}(P)$ denotes $\sum_{i=0}^{k-1}%
(\mathbf{1}^{\otimes i}\mathbf{\otimes}\Delta_{j}\otimes\mathbf{1}^{\otimes
k-i-1})\Delta_{k}(P).$
\end{itemize}

The fact that $\Delta_{j}$ and $\Delta_{k}$ vanish on edges and vertices when
$j,k\geq3$ implies:

\begin{proposition}
\label{compositions}$\Delta_{j}\Delta_{k}(P)=\partial\Delta_{k}(P)=0$ whenever
$i,j\geq3.$
\end{proposition}

In view of Proposition \ref{compositions}, all non-vanishing terms in Relation
\ref{relation} apply some $\Delta_{k}$ to the $2$-cell $P$. Therefore (up to
sign) Relation \ref{relation} reduces to
\begin{equation}
\left(  \Delta_{k-1}\Delta_{2}+\Delta_{2}\Delta_{k-1}+\partial\Delta
_{k}\right)  (P)=0. \label{general relation}%
\end{equation}
The signs in Relation \ref{general relation} follow from the \emph{Sign
Commutation Rule}: \emph{If an object of degree }$p$\emph{\ passes an object
of degree }$q$\emph{, affix the sign }$(-1)^{pq}$. First, consider a term
$e_{j_{1}}\otimes e_{j_{2}}\otimes\cdots\otimes e_{j_{k}}$ of $\Delta
_{k}\left(  P\right)  .$ Since $\left\vert \partial\right\vert =-1$,
multiplying by the sign in Relation \ref{relation} and simplifying gives
\[
\left(  -1\right)  ^{k-2}\left(  \mathbf{1}^{\otimes i-1}\mathbf{\otimes
}\partial\otimes\mathbf{1}^{\otimes k-i}\right)  \left(  e_{j_{1}}%
\otimes\cdots\otimes e_{j_{k}}\right)  =\left(  -1\right)  ^{i+k+1}e_{j_{1}%
}\otimes\cdots\otimes\partial e_{j_{i}}\otimes\cdots\otimes e_{j_{k}}.
\]
Second, since $\left\vert \Delta_{2}\right\vert =0,$ the sign of $\Delta
_{2}\Delta_{k-1}$ is the sign $\left(  -1\right)  ^{i+k},$ which is opposite
the sign in Relation \ref{relation}, where $\Delta_{2}$ is applied in the
$i^{th}$ position. Third, the signs of the terms $\left(  \Delta_{k-1}%
\otimes\mathbf{1}\right)  \Delta_{2}$ and $\left(  \mathbf{1}\otimes
\Delta_{k-1}\right)  \Delta_{2}$ given by Relation \ref{relation} simplify to
$-1$ and $\left(  -1\right)  ^{k+1},$ respectively. Since $\left\vert
\Delta_{k-1}\right\vert =k-3,$ the Sign Commutation Rule can only introduce a
sign when $\mathbf{1\otimes}\Delta_{k-1}$ is applied to a pair of edges.
However, this particular situation never occurs in our proof.

\begin{lemma}
All non-vanishing terms contain exactly one tensor factor $v_{j}$.
Furthermore, if $e_{i}\otimes v_{j}$ or $v_{j}\otimes e_{k}$ appears within
some term, then $i<j$ or $j\leq k$.
\end{lemma}

\begin{proof}
\textbf{Case 1: } Consider $\partial\Delta_{k}(P)$. Note that $\Delta_{k}(P)$
will either vanish if $k\geq n$, or it will produce terms of the form
$e_{i_{1}}\otimes\cdots\otimes e_{i_{k}},\text{ where }0<i_{1}<\cdots<i_{k}%
<n$. Applying $\partial$ to any $e_{i}$ will create two new terms by replacing
it with $v_{i}$ or $v_{i+1}$. Since the subscript $i$ is between the
subscripts on either side, either choice gives the desired result.

\textbf{Case 2: } Consider $\Delta_{k-1}\Delta_{2}(P)$. Since $\Delta_{k-1}$
acts non-trivially only on $P$, the terms of interest produced by $\Delta_{2}$
are $v_{1}\otimes P$ and $P\otimes v_{n}$. When $\Delta_{k-1}$ acts on the
$P$, it produces a term of the form $e_{i_{1}}\otimes\cdots\otimes e_{i_{k-1}%
},\text{ where }0<i_{1}<\cdots<i_{k-1}<n$. Since $0<i<n$ for all $i$, all
terms either begin with $v_{1}$ or end with $v_{n}$ and have the desired form.

\textbf{Case 3: } Consider $\Delta_{2}\Delta_{k-1}(P)$. Notice that
$\Delta_{k-1}(P)$ produces terms of the form $e_{i_{1}}\otimes\cdots\otimes
e_{i_{k-1}},\text{ where }0<i_{1}<\cdots<i_{k-1}<n$. When $\Delta_{2}$ is
applied to an $e_{i}$ factor, it has the effect of either inserting $v_{i}$ to
the left, or $v_{i+1}$ to the right of the $e_{i}$. Since the subscript $i$
is between the subscripts on either side, either choice gives the desired result.
\end{proof}

This information allows us to classify the terms in Relation
\ref{general relation} relative to the position of $v_{i}$ with respect to
$e_{i-1}$ and $e_{i}$. We say that $v_{i}$ is \textit{left adjacent} if
$e_{i-1}$ is immediately to its left; $v_{i}$ is \textit{right adjacent} if
$e_{i}$ is immediately to its right.

\begin{definition}
A term in which $v_{i}$ is left adjacent is \textbf{left attached}; a term in which $v_{i}$ is right adjacent is \textbf{right attached}. A term that is both left and right attached is \textbf{doubly attached}; a term that is left attached or right attached but not both is \textbf{singly attached}; and a term that is neither left attached nor right attached is \textbf{unattached}.
\end{definition}

We can sub-classify within these sets as follows:

\begin{definition}
A term that begins with $v_{1}$ is \textbf{extreme minimal}. A term that ends
with $v_{n}$ is \textbf{extreme maximal}.
\end{definition}

Since doubly attached terms cannot be extreme, there are five classes of
terms, and the proof reduces to showing that each class cancels itself in
Relation \ref{general relation}. But before we can do this, we need some lemmas.

\begin{lemma} 
\label{Lemma for part 1}
\textit{Let }$2<k<n.$\textit{ If }$v_{i}$\textit{
appears in the }$i^{th}$\textit{ position of a term in }$\partial\Delta
_{k}\left(  P\right)  ,$\textit{ that term is generated one of the following
two ways, each way contributing at most one instance:}

\begin{itemize}
\item \textbf{Way 1: }\textit{If the term is unattached on the left and is not
minimally extreme, an instance of that term is generated with sign }$\left(
-1\right)  ^{i+k+1}$\textit{.}

\item \textbf{Way 2: }\textit{If the term is unattached on the right and is
not maximally extreme, an instance of that term is generated with sign
}$\left(  -1\right)  ^{i+k}$\textit{.\medskip}
\end{itemize}
\end{lemma}

\begin{proof}
\textbf{Case 1:} Consider a term that contains $e_{i_{1}}\otimes v_{i_{2}%
}\otimes e_{i_{3}},\text{ where }0<i_{1}<i_{2}\leq i_{3}<n$. Since $\Delta
_{k}(P)$ cannot generate $v_{i_{2}}$, it was produced by $\partial$. Since
$v_{i_{2}}$ appears between $e_{i_{1}}$ and $e_{i_{3}}$, we know that
$\partial$ was applied to a term of the form $e_{i_{1}}\otimes e_{x}\otimes
e_{i_{3}}$ with $0<i_{1}<x<i_{3}<n$. Furthermore, there are only two possible
edges $e_{x}$ that $\partial$ could have acted upon to generate $v_{i_{2}}$,
namely when $x=i_{2}-1$ or $x=i_{2}$. Therefore, since there are only two
possible values for $x$, there are at most two ways to generate any given term.

Suppose that $x=i_{2}-1$. Since $i_{1}<x$, we have $i_{1}<i_{2}-1$ forcing the
original term to be unattached on the left. To see that all left unattached
terms can be generated, simply note that the subscript of any left unattached
term will satisfy the inequality $i_{1}<i_{2}-1$ by definition. Then
$\Delta_{k}(P)$ produces the term $e_{i_{1}}\otimes e_{i_{2}-1}\otimes
e_{i_{3}},$ where $0<i_{1}<i_{2}\leq i_{3}<n$, and produces the desired term
after $\partial$ is applied to $e_{i_{2}-1}$. Now applying $\partial$ this way
produces a positive $v_{i_{2}}$ in the same position in which $\partial$ was
applied, and by the discussion above regarding the Sign
Commutation Rule, the term's sign is
$(-1)^{i+k+1}$. Similarly, suppose
$x=i_{2}$. Then all right unattached terms, and only  right unattached terms, can be
generated. Similar analysis reveals that the term's sign is $(-1)^{i+k}$.

\textbf{Case 2: } Consider a term of the form $v_{i_{2}}\otimes e_{i_{3}%
}\otimes\cdots,\text{ where }0<i_{2}\leq i_{3}<n$. By an argument similar to
the one above, $\partial$ must have been applied to a term of the form
$e_{x}\otimes e_{i_{3}}\otimes\cdots$, where $x=i_{2}-1$ or $x=i_{2}$. The
proof is similar to the case above and left to the reader; the only significant
difference is that when $x=i_{2}-1$ we cannot immediately conclude that all
possible left unattached terms will be generated. This is because $x>0$, and
since $x=i_{2}-1$, we have $i_{2}>1$. Hence only non-minimally extreme left
unattached terms can be formed, and an analysis similar to that in Case 1
shows that the terms produced have the sign $(-1)^{i+k+1}$.

\textbf{Case 3: } Consider a term of the form $\cdots\otimes e_{i_{1}}\otimes
v_{i_{2}},\text{ where }0<i_{1}<i_{2}<n$. This case is can be proved using an
argument very similar to case 2, the details of which are left to the reader.
\end{proof}

\begin{lemma} \label{Lemma for part 3}
\textit{Let }$3<k\leq n.$\textit{ If }$v_{i}%
$\textit{ appears in the }$i^{th}$\textit{ position of a term in }$\Delta
_{2}\Delta_{k-1}\left(  P\right)  ,$\textit{ that term is generated one of the
following two ways, each way contributing at most one instance:}

\begin{itemize}
\item \textbf{Way 1: }\textit{If the term is attached on the left, an instance
of that term is generated with sign }$\left(  -1\right)  ^{i+k+1}$\textit{.}

\item \textbf{Way 2: }\textit{If the term is attached on the right, an
instance of that term is generated with sign }$\left(  -1\right)  ^{i+k}%
$\textit{.\bigskip}
\end{itemize}
\end{lemma}

\begin{proof}
\textbf{Case 1: } Consider a term that contains $e_{i_{1}}\otimes v_{i_{2}%
}\otimes e_{i_{3}},\text{ where }0<i_{1}<i_{2}\leq i_{3}<n$. Since
$\Delta_{k-1}(P)$ cannot produce $v_{i_{2}}$, it is produced by $\Delta_{2}$,
and because $\Delta_{2}$ inserts a $v$ next to the $e$ to which it was
applied, $\Delta_{2}$ was applied to either $e_{i_{1}}$ or $e_{i_{3}}$.
Additionally, because of how $\Delta_{2}$ is defined, the only way $\Delta
_{2}(e_{i_{1}})$ can generate $v_{i_{2}}$ is if $i_{1}=i_{2}-1$, and the only
way $\Delta_{2}(e_{i_{3}})$ can generate $v_{i_{2}}$ is if $i_{2}=i_{3} $.
Therefore, there are at most two ways to create the desired term.

Now suppose $\Delta_{2}$ was applied to $e_{i_{1}}$ and that $i_{1}=i_{2}-1$.
Then by definition, the resulting term will be left attached. Furthermore, any
given left attached term of the form $e_{i_{2}-1}\otimes v_{i_{2}}\otimes
e_{i_{3}},$ where $0<i_{2}-1\leq i_{3}<n,$ can be generated by simply applying
$\Delta_{2}$ to the term $e_{i_{2}-1}\otimes e_{i_{3}}$ generated by
$\Delta_{k}(P)$. Now in order to create a term with $v_{i_{2}}$ in the
$i^{th}$ position, we must apply $\Delta_{2}$ to the $(i-1)^{st}$ position,
and so by the discussion above on the Sign Commutation Rule, the term's
sign is $(-1)^{i+k+1}$. If we suppose that $\Delta_{2}$ was
applied to $e_{i_{3}}$ and that $i_{2}=i_{3}$, then a very similar analysis
shows that all right attached terms of the form $e_{i_{1}}\otimes v_{i_{2}%
}\otimes e_{i_{2}},$ where $0<i_{2}<n$ are generated with sign
of $(-1)^{i+k}$.

\textbf{Case 2: } Consider a term of the form $v_{i_{2}}\otimes e_{i_{3}%
}\otimes\cdots,\text{ where }0<i_{2}<i_{3}<n$. By an argument similar to the
one above, this term can only be generated if $\Delta_{2}$ is applied to
$e_{i_{3}}$ and if $i_{2}=i_{3}$. The same reasoning as in case one shows that
all right attached terms, and only those terms, will be generated with sign
$(-1)^{i+k}$.

\textbf{Case 3: } Consider a term of the form $\cdots\otimes e_{i_{1}}\otimes
v_{i_{2}},\text{ where }0<i_{1}<i_{2}<n$. This case is nearly identical to
case 2, and we leave it to the reader to show that all left attached terms,
and only those terms, can be generated with sign $(-1)^{i+k+1}$.
\end{proof}

\begin{proposition}
\label{Part 1}Let $2<k<n.$ If $v_{i}$ appears in the $i^{th}$ position of a
term in $\partial\Delta_{k}(P),$ after cancellations the terms that remain are:

\begin{itemize}
\item All maximally extreme unattached terms with positive sign.

\item All minimally extreme unattached terms with sign $(-1)^{k}$.

\item All left attached terms with sign $(-1)^{i+k}$.

\item All right attached terms with sign $(-1)^{i+k+1}$.
\end{itemize}
\end{proposition}

\begin{proof}
The proof requires several straightforward applications of Lemma
\ref{Lemma for part 1}, some of which show that no doubly attached terms are
generated, all unattached terms generated cancel, and all extreme terms
generated are unattached. The details are left to the reader.
\end{proof}

\begin{proposition}
\label{Part 2}Let $3<k\leq n$. Then $\Delta_{k-1}\Delta_{2}(P)$ contains:

\begin{itemize}
\item All minimally extreme terms with sign $(-1)^{k+1}$.

\item All maximally extreme terms with negative sign.
\end{itemize}
\end{proposition}

\begin{proof}
Since $\Delta_{k-1}$ only acts non-trivially on $P$, the only terms produced
by $\Delta_{2}$ that do not immediately vanish are $v_{1} \otimes P + P
\otimes v_{n}$. Additionally, since $k>3$, $\Delta_{k-1}$ produces
no primitive terms when it is applied by the definition of $\Delta_{k}$ for
$k>2$.

When $\Delta_{k-1}$ is applied to the first term, it generates terms of the
form $v_{1} \otimes e_{i_{1}} \otimes e_{i_{2}} \otimes\cdots$, which are all
minimally extreme. Since the sign commutation rule introduces nothing new, we have that the term has the final sign $(-1)^{k+1}$.

When $\Delta_{k-1}$ is applied to the second term, it generates terms of the
form $\cdots\otimes e_{i_{k-3}}\otimes e_{i_{k-2}}\otimes v_{n}$, which are
maximally extreme. Since the sign commutation rule introduces nothing new, we have that the sign will always be negative.
\end{proof}

\begin{proposition}
\label{Part 3} Let $3<k\leq n$. If $v_{i}$ appears in the $i^{th}$ position of
a term in $\Delta_{2}\Delta_{k-1}(P)$, after cancellations the terms that
remain are:

\begin{itemize}
\item All maximally extreme singly attached terms with positive sign.

\item All minimally extreme singly attached terms with sign $(-1)^{k}$.

\item All left attached terms with sign $(-1)^{i+k+1}$.

\item All right attached terms with sign $(-1)^{i+k}$.
\end{itemize}
\end{proposition}

\begin{proof}
The proof requires several straightforward applications of Lemma
\ref{Lemma for part 3} some of which show that no unattached terms are
generated, all doubly attached terms cancel, and all extreme terms are singly
attached. The details are left to the reader.
\end{proof}

Now when $k>3$, all three propositions apply and we see that in Relation
\ref{general relation} the singly attached terms in $\partial\Delta_{k}(P)$
and $\Delta_{2}\Delta_{k-1}(P)$ cancel, the extreme unattached terms in
$\partial\Delta_{k}(P)$ and $\Delta_{k-1}\Delta_{2}(P)$ cancel, and the
extreme singly attached terms in $\Delta_{k-1}\Delta_{2}(P)$ and $\Delta
_{2}\Delta_{k-1}(P)$ cancel. Therefore, all terms cancel and the relation is satisfied.

Having proved the theorem for $k>3,$ we consider the special case $k=3.$ We
must show that:
\[
\lbrack(\partial\otimes\mathbf{1}\otimes\mathbf{1}+\mathbf{1}\otimes
\partial\otimes\mathbf{1}+\mathbf{1}\otimes\mathbf{1}\otimes\partial
)\Delta_{3}+(-\Delta_{2}\otimes\mathbf{1}+\mathbf{1}\otimes\Delta_{2}%
)\Delta_{2}](P)=0
\]

\begin{proposition}
The terms generated by $(\partial\Delta_{3}+\Delta_{2}\Delta_{2})(P)$ form
three classes, which independently satisfy the conclusions of Propositions
\ref{Part 1}, \ref{Part 2}, and \ref{Part 3}.
\end{proposition}

\begin{proof}
Note that the proof of Proposition \ref{Part 1} still applies to
$\partial\Delta_{3}(P)$ as before since it never used the restriction that
$k\geq3$. Therefore, we must show that $(\Delta_{2}\otimes\mathbf{1}%
-\mathbf{1}\otimes\Delta_{2})\Delta_{2}$ can be decomposed into parts that
respectively satisfy the conclusions of Proposition \ref{Part 2} and
Proposition \ref{Part 3}. Now expanding we have
\begin{align*}
&  (-\Delta_{2}\otimes\mathbf{1}+\mathbf{1}\otimes\Delta_{2})\Delta
_{2}(P)=(-\Delta_{2}\otimes\mathbf{1}+\mathbf{1}\otimes\Delta_{2})\left(
v_{1}\otimes P+P\otimes v_{n}\right) \\
&  \hspace*{2in}(-\Delta_{2}\otimes\mathbf{1}+\mathbf{1}\otimes\Delta_{2}%
)\sum_{0<i_{1}<i_{2}<n}e_{i_{1}}\otimes e_{i_{2}}.
\end{align*}
Notice that applications of $\displaystyle(-\Delta_{2}\otimes\mathbf{1}%
+\mathbf{1}\otimes\Delta_{2})\sum_{0<i_{1}<i_{2}<n}e_{i_{1}}\otimes e_{i_{2}}$
produce terms of the form $e_{i_{1}}\otimes e_{i_{2}}\otimes\cdots\otimes
e_{i_{k-1}},\text{ where }0<i_{1}<i_{2}<\cdots<i_{k-1}<n$, then applies
$\Delta_{2}$ to one of them. This is exactly the set up in Proposition \ref{Part 3}. Since the application
of Lemma \ref{Lemma for part 3} in the proof of Proposition \ref{Part 3} requires the condition $k>3$, and since Lemma \ref{Lemma for part 3} only uses the condition $k>3$ to ensure that $\Delta
_{k-1}\left(  P\right)  $ generates no primitive terms, the proof of
Proposition \ref{Part 3} also applies to $(\Delta_{2}\otimes\mathbf{1}+\mathbf{1\otimes \Delta}_{2})\sum_{0<i_{1}<i_{2}<n}e_{i_{1}}\otimes e_{i_{2}}$ since all terms of $\mathbf{\Delta}_{2}\left(  P\right)  $ except its primitive terms are present. Thus we obtain the same types of terms as in Proposition \ref{Part 3}

Therefore, we must show that what remains behaves as in the general case of
$\Delta_{k-1}\Delta_{2}$ and produces all possible extreme terms. By expanding, one can show that
\begin{align*}
&  (-\Delta_{2}\otimes\mathbf{1}+\mathbf{1}\otimes\Delta_{2})\left(
v_{1}\otimes P+P\otimes v_{n}\right)  =\\
&  \hspace*{1in}-\sum_{0<i_{1}<i_{2}<n}e_{i_{1}}\otimes e_{i_{2}}\otimes
v_{n}+\sum_{0<i_{1}<i_{2}<n}v_{1}\otimes e_{i_{1}}\otimes e_{i_{2}},
\end{align*}
which generates every possible extreme term and only extreme terms, all with
the correct sign for $k=3$. Therefore, $(\Delta_{2}\otimes\mathbf{1}%
+\mathbf{1}\otimes\Delta_{2})\left(  v_{1}\otimes P+P\otimes v_{n+1}\right)  $
satisfies the hypothesis of Proposition \ref{Part 2}.
\end{proof}

Since the terms generated by $(\partial\Delta_{3}+\Delta_{2}\Delta_{2})(P)$
fall into classes that individually satisfy the conclusions of Proposition
\ref{Part 1}, Proposition \ref{Part 2}, and Proposition \ref{Part 3}, they
cancel each other out as before. Therefore $\Delta_{3}$ also satisfies
Relation \ref{relation} and the proof of Theorem \ref{Main Result} is complete.

\section{Generalization of Result}

Until now, we have been working with an $n$-gon $P$ whose initial and terminal
vertices are adjacent. Our result extends to situations in which the initial
and terminal vertices are non-adjacent. Suppose $v_{1}$ is the initial vertex
and $v_{t}$ is the terminal vertex, where $t>1$. Define the generalized
$\Delta_{2}^{\prime}$ to be the same as in the introduction, except that
\[
\Delta_{2}^{\prime}(P)=v_{1}\otimes P+P\otimes v_{t}+\sum_{0<i_{1}<i_{2}%
<t}e_{i_{1}}\otimes e_{i_{2}}-\sum_{n\geq i_{1}>i_{2}\geq t}e_{i_{1}}\otimes
e_{i_{2}}.
\]
Additionally, for $k>2$, define the generalized $k$-ary $A_{\infty}$-coalgebra
operation $\Delta_{k}^{\prime}\ $by
\[
\Delta_{k}^{\prime}(\sigma)=\left\{  \hspace{-.3in}
\begin{array}
[c]{cc}%
\sum\limits_{0<i_{1}<i_{2}<\cdots<i_{k}<t}e_{i_{1}}\otimes e_{i_{2}}%
\otimes\cdots\otimes e_{i_{k}}\smallskip & \\
\hspace{.5in}-\sum\limits_{n\geq i_{1}>i_{2}>\cdots>i_{k}\geq t}e_{i_{1}%
}\otimes e_{i_{2}}\otimes\cdots\otimes e_{i_{k}}, & \text{if }\sigma=P,\\
& \\
0, & \text{otherwise.}%
\end{array}
\right.
\]
Then by definition, $\Delta_{k}^{\prime}=0$ for all $k\geq max\{t,n-t+2\}.$
Now all that remains is to show that the operations $\{\Delta_{n}^{\prime
}\}_{n\geq2}$ extend to this general setting.

\begin{corollary}
\label{cor} Let $P$ be an $n$-gon with initial vertex $v_{1}$ and terminal
vertex $v_{t},$ where $t>1$. The operations $\left\{  \Delta_{k}%
^{\prime}\right\}  _{2\leq k<max\{t,n-t+2\}}$ defined above impose an
$A_{\infty}$-coalgebra structure on $\left(  C_{\ast}\left(  P\right)
,\partial\right)  .$
\end{corollary}

\begin{proof}
Draw an additional edge from $v_{1}$ to $v_{t}$ and denote it by $e_{0}.$
Define $P_{1}$ to be the polygon with vertices $v_{1}, v_{2}, \ldots, v_{t}$
oriented counterclockwise and let $P_{2}$ to be the polygon with vertices
$v_{t}, v_{t+1}, \ldots, v_{n}, v_{1}$ oriented counterclockwise as
illustrated in Figure 4.

\begin{center}
\hspace{.7in} \includegraphics[height = 5cm]{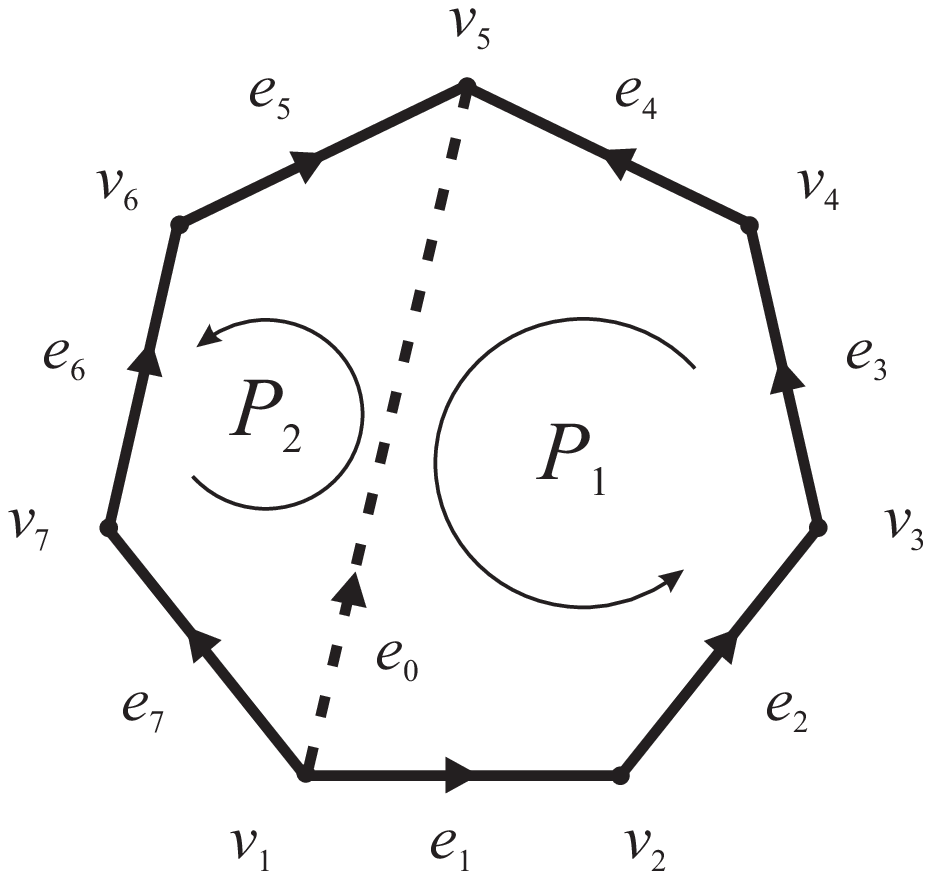} \newline Figure 4.
A $7$-gon with $v_{t} = v_{5}.$
\end{center}

Then by the way edges are directed with respect to the orientation, we have
$\partial(P_{1})=e_{1}+e_{2}+\cdots+e_{t-1}-e_{0}$ and $\partial(P_{2}%
)=-e_{t}-e_{t+1}-\cdots-e_{n}+e_{0}.$ Note that $\partial(P_{1})+\partial
(P_{2})=\partial(P).$ Furthermore, define $v_{1}$ to be the initial vertex and
$v_{t}$ to be the terminal vertex in both $P_{1}$ and $P_{2};$ then $P_{1}$
and $P_{2}$ satisfy the hypothesis of Theorem \ref{Main Result} and it is
straightforward to show that $\Delta_{2}(P_{1})+\Delta_{2}%
(P_{2})=\Delta_{2}^{\prime}(P)$ and $\Delta_{k}(P_{1})+\Delta_{k}%
(P_{2})=\Delta_{k}^{\prime}(P)$ for $k>2$. All that remains is to verify that
Relation \ref{relation} holds on $P$ for all $k\geq2$, which can be done using
the relations above and applying the main theorem to each part. The details
are left to the reader. Therefore, the operations $\{\Delta_{n}^{\prime
}\}_{n\geq2}$ defined on $P$ above satisfy all $A_{\infty}$-coalgebra
relations on cellular chains of $P$.
\end{proof}

\section{Application to closed compact surfaces}

The celebrated classification of closed compact surfaces (cf. \cite{Massey},
for example) states that a closed compact surface of genus $g,$ denoted by
$X_{g},$ is homeomorphic to a sphere with $g\geq 0$ handles when 
orientable or a connected sum of $g\geq 1$ real projective planes when
unorientable. 

To obtain the connected sum $X\#Y$ of two surfaces $X$ and $Y,$ remove the interior
of a disk from $X$ and from $Y$ then glue the two surfaces together along their boundaries. Of course, a sphere with $g\geq 1$ handles 
is the connected sum of $g$ tori, and a Klein bottle is the connected sum of two real projective planes.

\begin{center}

\centering
\includegraphics[
height=2.0in,
width=4.0in
]{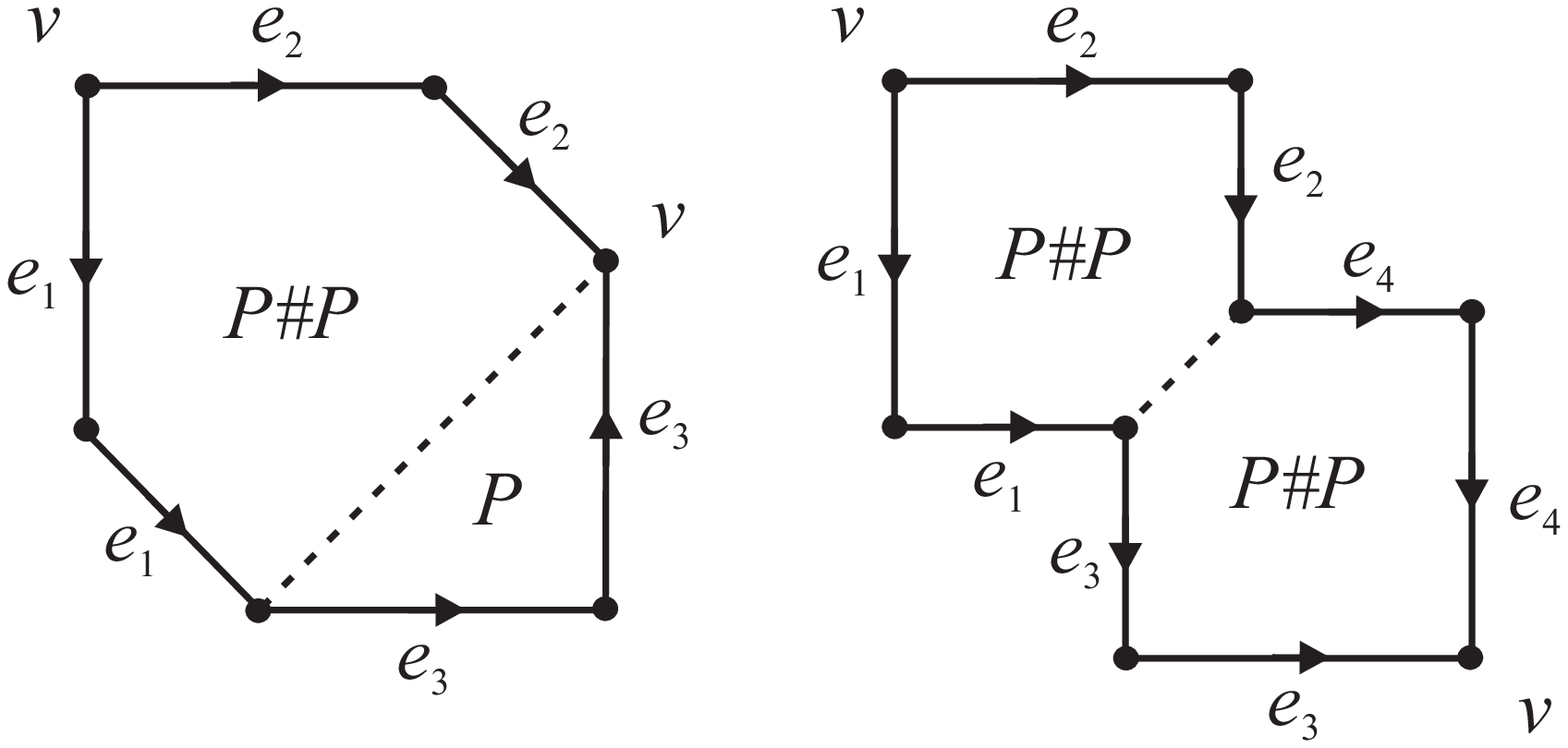}
%\end{figure}
%EndExpansion

Figure 5. A polygonal decomposition of the connected sum of three and four real projective planes.

\medskip

\centering
\includegraphics[
height=2.6541in,
width=2.6342in
]{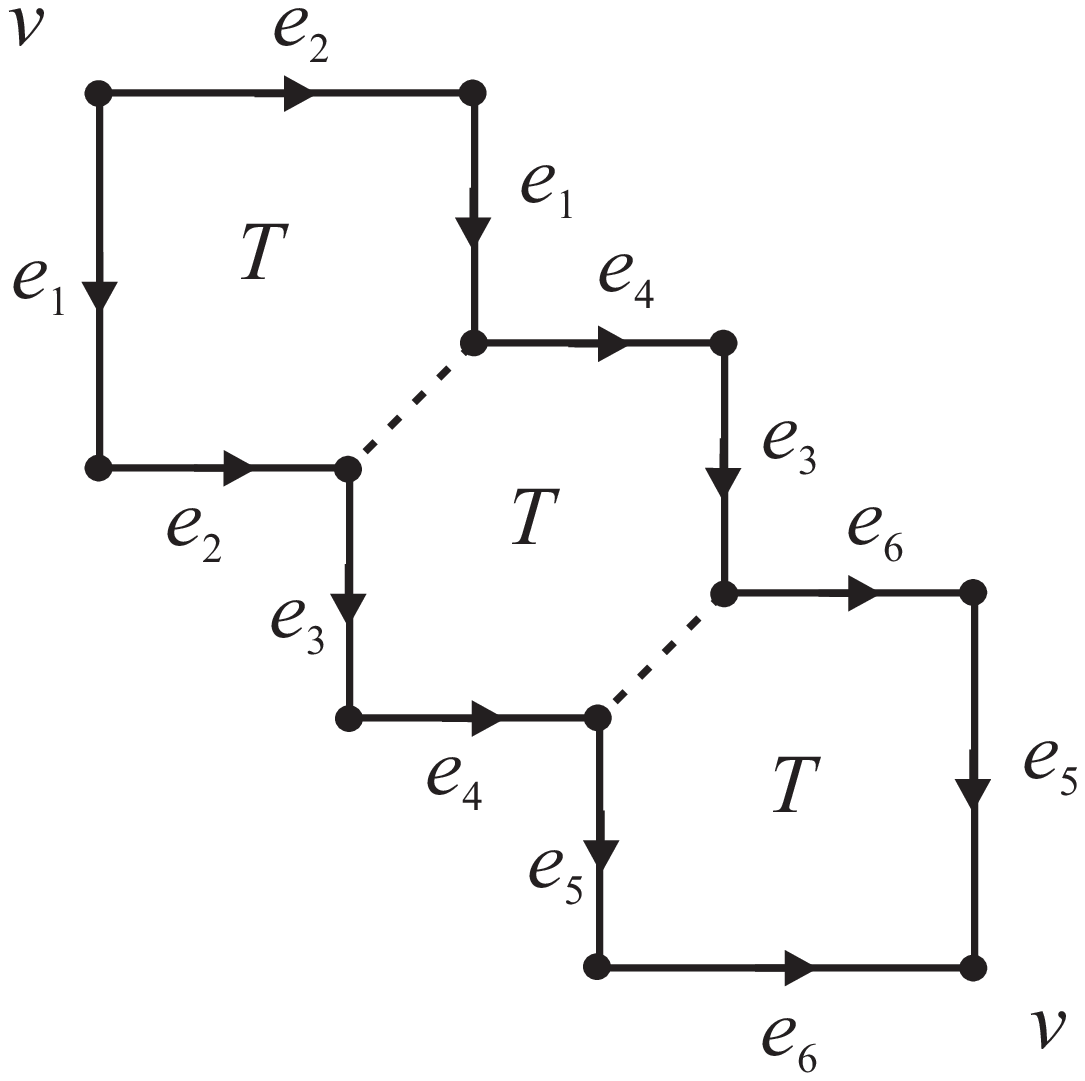}

Figure 6. A polygonal decomposition of the connected sum of three tori.\bigskip
\end{center}

Furthermore, when $g\geq 1,$ $X_{g}$ can be expressed
as the quotient of a $4g$-gon when orientable or
a $2g$-gon when unorientable as pictured in Figures 5 and 6 (the dotted lines represent common boundaries in the connected sums).  To recover $X_g$ from a polygonal decomposition $P_g$, glue the edges with the same
label together so that the arrows directing the edges align. This gluing operation defines the projection $p:P_g\to X_g$, which encodes the topology of $X_g$.

The cutting and pasting procedures indicated in Figures 5 and 6 can be
continued indefinitely, and the directed edges of a particular polygonal decomposition $P_{g}$
so obtained define a vertex poset with initial and terminal vertices labeled
$v$ (all vertices are identified in $X_g$). These are exactly the configurations to
which Corollary \ref{cor} applies.

Indeed, the formal $A_\infty$-coalgebra structure on $C_*(P_g)$ given by Corollary \ref{cor} projects to a quotient
structure on $C_{\ast}\left(  X_{g}\right)  $ in the obvious way. If $X_3$ is the connected sum of three real projective planes, for example, using the decomposition in Figure 5 we obtain an $A_{\infty}$-coalgebra structure on $C_{\ast
}\left(  X_3\right)  $ with  non-trivial operations
\begin{align*}
\Delta_{2}\left(  X_3\right)   &  =v\otimes X_3+X_3\otimes v+e_{1}\otimes
e_{1}-e_{2}\otimes e_{2}+e_{3}\otimes e_{3}\\
\Delta_{4}\left(  X_3\right)   &  =e_{1}\otimes e_{1}\otimes e_3\otimes e_{3}.
\end{align*}
Furthermore, each cellular chain in
$ C_{\ast}\left(  P_{g}\right)\otimes\mathbb{Z}_2  $ projects to a non-bounding cycle in $ C_{\ast}\left(  X_{g}\right)\otimes\mathbb{Z}_2  $ so that 
$H_{\ast}\left(  X_{g};\mathbb{Z}_{2}\right)  =C_{\ast}\left(
X_{g}\right)  \otimes\mathbb{Z}_{2}$.  

For a general unorientable $X_{g}$, label the edge of the $i^{th}$ real
projective plane $e_{i}$, and label the vertex $v.$ Let $\left\lfloor
x\right\rfloor $ denote the floor of $x$, and for a given $s$ define the
sequences
\[
\left\{  i_{p}=2\left\lfloor \tfrac{p+1}{2}\right\rfloor -1\right\}
_{p=1}^{2s}\text{ and }\left\{  j_{q}=2\left\lfloor \tfrac{q+1}{2}%
\right\rfloor \right\}  _{q=1}^{2s}.
\]

\begin{corollary}
\label{unorientable} Let $X_{g}$ be a closed compact unorientable surface of
genus $g\geq2$ and let $P_{g}$ be the polygonal decomposition of $X_{g}$
indicated in Figure 5. The formal $A_{\infty}$-coalgebra structure on
$C_{*}(P_{g})$ projects to a non-trivial $A_{\infty}$-coalgebra structure on
$C_{\ast}\left(  X_{g}\right)  $ with operations $\left\{  \Delta_{k}\right\}
_{k\geq2}$ determined by the quotient map
$q:P_{g}\to X_{g}$ and defined by\medskip

$\Delta_{2}\left(  v\right)  =v\otimes v$\medskip

$\Delta_{2}\left(  e_{i}\right)  =v\otimes e_{i}+e_{i}\otimes v$

$\Delta_{2}(X_{g})=v\otimes X_{g}+X_{g}\otimes v+%
%TCIMACRO{\dsum \limits_{\substack{i=1\\g\in\left\{  2s-1,2s\right\}  }}^{s}}%
%BeginExpansion
{\displaystyle\sum\limits_{\substack{i=1\\g\in\left\{  2s-1,2s\right\}  }%
}^{s}}
%EndExpansion
e_{2i-1}\otimes e_{2i-1}-%
%TCIMACRO{\dsum \limits_{\substack{j=1\\g\in\left\{  2s,2s+1\right\}  }}^{s}}%
%BeginExpansion
{\displaystyle\sum\limits_{\substack{j=1\\g\in\left\{  2s,2s+1\right\}  }%
}^{s}}
%EndExpansion
e_{2j}\otimes e_{2j}$\medskip

$\Delta_{k}(X_{g})=%
%TCIMACRO{\dsum \limits_{\substack{0< p_{1}<\cdots<p_{k}\leq2s\\g\in\left\{
%2s-1,2s\right\}  }}}%
%BeginExpansion
{\displaystyle\sum\limits_{\substack{0< p_{1}<\cdots<p_{k}\leq2s\\g\in\left\{
2s-1,2s\right\}  }}}
%EndExpansion
e_{i_{p_{1}}}\otimes\cdots\otimes e_{i_{p_{k}}}-%
%TCIMACRO{\dsum \limits_{\substack{0< q_{1}<\cdots<q_{k}\leq2s\\g\in\left\{
%2s,2s+1\right\}  }}}%
%BeginExpansion
{\displaystyle\sum\limits_{\substack{0< q_{1}<\cdots<q_{k}\leq2s\\g\in\left\{
2s,2s+1\right\}  }}}
%EndExpansion
e_{j_{q_{1}}}\otimes\cdots\otimes e_{jq_{k}},$ $k\geq3$\medskip

$\Delta_{k}(\sigma)=0,\,\ $for all $\sigma\neq X_{g}$ and $k\geq3.$\medskip
\end{corollary}

For a general orientable $X_{g}$, label the edges along one edge-path from
initial to terminal vertex $e_{1},e_{2},\ldots,e_{2g}.$ Let $\hat{e}%
_{i_{2k-1}}=e_{i_{2k}}$ and $\hat{e}_{i_{2k}}=e_{i_{2k}-1},$ and label the
edges along the other edge-path $\hat{e}_{1},\hat{e}_{2},\ldots,\hat{e}_{2g}.$

\begin{corollary}
\label{orientable}
Let $X_{g}$ be a closed compact orientable surface of genus $g\geq2$ and let
$P_{g}$ be the polygonal decomposition of $X_{g}$ indicated in Figure 6. The
formal $A_{\infty}$-coalgebra structure on $C_{*}(P_{g})$ projects to a
non-trivial $A_{\infty}$-coalgebra structure on $C_{\ast}\left(  X_{g}\right)
$ with operations $\left\{  \Delta_{k}\right\}  _{k\geq2}$ determined by the
the quotient map $q:P_{g}\to X_{g}$ and defined
by\medskip

$\Delta_{2}\left(  v\right)  =v\otimes v$\medskip

$\Delta_{2}\left(  e_{i}\right)  =v\otimes e_{i}+e_{i}\otimes v$

$\Delta_{2}(X_{g})=v\otimes X_{g}+X_{g}\otimes v+%
%TCIMACRO{\dsum \limits_{i=1}^{g}}%
%BeginExpansion
{\displaystyle\sum\limits_{i=1}^{g}}
%EndExpansion
e_{2i-1}\otimes e_{2i}-e_{2i}\otimes e_{2i-1}$\medskip

$\Delta_{k}(X_{g})=%
%TCIMACRO{\dsum \limits_{0<i_{1}<\cdots<i_{k}\leq2g}}%
%BeginExpansion
{\displaystyle\sum\limits_{0<i_{1}<\cdots<i_{k}\leq2g}}
%EndExpansion
e_{i_{1}}\otimes\cdots\otimes e_{i_{k}}-\hat{e}_{i_{1}}\otimes\cdots
\otimes\hat{e}_{i_{k}},$ $k\geq3$\medskip

$\Delta_{k}(\sigma)=0,$ for all $\sigma\neq X_{g}$ and $k\geq3.$\medskip
\end{corollary}

\begin{remark}
It is interesting to note that our definitions of $\Delta_2$ in Corollaries \ref{unorientable} and \ref{orientable} allow us to read off the cup product on $H_*\left(X_g;\mathbb{Z}_2\right)$ directly from the components of $\Delta_2$ without performing additional calculations. In general, one has a choice: Compute cup products using a standard diagonal at the expense of long calculations or construct an application-specific diagonal that minimizes the calculations at the expense of the accompanying combinatorial difficulties.  
\end{remark}

Finally, when $X_g$ is the sphere $S^2$, the $A_{\infty}%
$-coalgebra structure on $C_{\ast}\left( X_g\right)  $ is clearly degenerate; if $X_g$ is a real projective plane, a torus, or a Klein bottle, there is one non-vanishing $A_{\infty}$-coalgebra operation on $C_{\ast}\left(  X_g \right)$, namely $\Delta_2$, which induces the non-trivial cup product in cohomology (the formula for $\Delta_2$ on a real projective plane requires independent verification; the proof is straight-forward and omitted). But if $X_g$ is a closed compact surface that is orientable with $g\geq2$ or unorientable with $g\geq3$,
the combinatorial homotopy coassociative diagonal on $C_*(P_g)$ induces a formal $A_{\infty}$-coalgebra structure that projects to a strictly coassociative  $A_{\infty}$-coalgebra structure on $C_{\ast}\left( X_g \right)$ with non-trivial higher order structure determined by the quotient map $q$. But whether or not the $A_{\infty}$-coalgebra structure
$(H=H_{\ast}\left(  X_{g};\mathbb{Z}_{2}\right)  ,\Delta_{k})$ observed here
is topologically invariant is an open question.

\end{document}